\documentclass[10pt]{amsart}
\usepackage{amssymb,amsthm,amsmath,amsfonts}
\usepackage{natbib}
\usepackage{hyperref}
\usepackage{enumerate}

\usepackage{color}

\makeatletter
\g@addto@macro\th@plain{\thm@headpunct{}}
\makeatother

\newtheorem{theorem}{Theorem}[section]

\newtheorem{lemma}[theorem]{Lemma}
\newtheorem{remark}[theorem]{Remark}

\newcommand{\xx}{ {\textbf x} }
\newcommand{\ab}{ {\textbf a} }
\newcommand{\bb}{ {\textbf b} }
\newcommand{\cb}{ {\textbf c} }

\newcommand{\yy}{ {\textbf y} }

\newcommand{\zz}{ {\textbf z} }
\newcommand{\ee}{ {\textbf e} }
\newcommand{\ub}{ {\textbf u} }
\newcommand{\vb}{ {\textbf v} }

\newcommand{\VV}{ \Omega }
\newcommand{\RR}{\mathbb{R}}
\newcommand{\KK}{\mathbb{K}}

\newcommand{\LL}{\mathbb{L}}
\newcommand{\PP}{\mathbb{P}}
\newcommand{\En}{\mathbb{E}}

\newcommand{\DD}{\mathcal{D}}

\newcommand{\TT}{\mathcal{T}}

\newcommand{\tr}{\mathrm{tr}\,}
\newcommand{\Trace}{\mathrm{Trace}\,}
\newcommand{\DDet}{\mathrm{Det}}

\providecommand{\norm}[1]{\lVert#1\rVert}

\providecommand{\scalar}[1]{\left\langle#1\right\rangle}

\newcommand{\gw}{ g }
\newcommand{\ww}{ w }

\title[Lukacs-Olkin-Rubin theorem without invariance of the ``quotient'']{The Lukacs-Olkin-Rubin theorem on symmetric cones without invariance of the ``quotient''}
\author[B. Ko\l{}odziejek]{Bartosz Ko\l{}odziejek}
\address{Faculty of Mathematics and Information Science\\Warsaw University of Technology\\Pl. Politechniki 1\\00-661 Warszawa, Poland}
\email{kolodziejekb@mini.pw.edu.pl}

\begin{document}

\begin{abstract}
We prove the Lukacs-Olkin-Rubin theorem without invariance of the distribution of the ``quotient'', which was the key assumption in the original proof of [Olkin--Rubin, Ann. Math. Stat. \textbf{33} (1962), 1272--1280]. Instead we assume existence of strictly positive continuous densities of respective random variables. We consider the (cone variate) ``quotient'' for any division algorithm satisfying some natural conditions. For that purpose, the new proof of the Olkin--Baker functional equation on symmetric cones is given.
\end{abstract}
\maketitle

\section{Introduction}
The \cite{Lukacs1955} theorem is one of the most celebrated characterizations of probability distributions. It states that \textit{if $X$ and $Y$ are independent, positive, non-degenerate random variables such that their sum and quotient are also independent then $X$ and $Y$ have gamma distributions with the same scale parameter}.

This theorem has many generalizations. The most important in the multivariate setting were given by \cite{OlRu1962} and \cite{CaLe1996}, where the authors extended characterization to matrix and symmetric cones variate distributions, respectively. There is no unique way of defining the quotient of elements of the cone of positive definite symmetric matrices $\VV_+$ and in these papers the authors have considered very general form $U=g(X+Y)\cdot X\cdot g^T(X+Y)$, where $g$ is the so called division algorithm, that is, $g(\ab)\cdot\ab\cdot g^T(\ab)=I$ for any $\ab\in\VV_+$, where $I$ is the identity matrix and $g(\ab)$ is invertible for any $\ab\in\VV_+$ (later on, abusing notation we will write $g(\xx)\yy=g(\xx)\cdot\yy\cdot g^{T}(\xx)$, that is, in this case $g(\xx)$ denotes the linear operator acting on $\VV_+$). The drawback of their extension was the additional strong assumption of invariance of the distribution of $U$ under a group of automorphisms. This result was generalized to homogeneous cones in \cite{BoHaMa2011}.

There were successful attempts in replacing the invariance of the ``quotient'' assumption with the existence of regular densities of random variables $X$ and $Y$. \cite{BW2002} assuming existence of strictly positive, twice differentiable densities  proved a characterization of Wishart distribution on the cone $\VV_+$ for division algorithm $g_1(\ab)=\ab^{-1/2}$, where $\ab^{1/2}$ denotes the unique positive definite symmetric root of $\ab\in\VV_+$. This results was generalized to all non-octonion symmetric cones of rank greater than $2$ and to the Lorentz cone for strictly positive and continuous densities by \cite{Kolodzeng,Kolo2013}.

Exploiting the same approach, with the same technical assumptions on densities as in \cite{BW2002} it was proven by \cite{HaLaZi2008} that the independence of $X+Y$ and the quotient defined through the Cholesky decomposition, i.e. $g_2(\ab)=T_\ab^{-1}$, where $T_\ab$ is a lower triangular matrix such that $\ab=T_\ab\cdot T_\ab^T\in\VV_+$, characterizes a wider family of distributions called Riesz (or sometimes called Riesz-Wishart). This fact shows that the invariance property assumed in \cite{OlRu1962} and \cite{CaLe1996} is not of technical nature only. Analogous results for homogeneous cones were obtained by \cite{Bo2005,Bout2009}.

In this paper we deal with the density version of Lukacs-Olkin-Rubin theorem on symmetric cones for division algorithm satisfying some natural properties. We assume that the densities of $X$ and $Y$ are strictly positive and continuous. We consider quotient $U$ for an arbitrary, fixed division algorithm $g$ as in the original paper of \cite{OlRu1962}, additionally satisfying some natural conditions. In the known cases ($g=g_1$ and $g=g_2$) this improves the results obtained in \cite{BW2002,HaLaZi2008,Kolo2013}. In general case, the densities of $X$ and $Y$ are given in terms of, so called, $w$-multiplicative Cauchy functions, that is functions satisfying 
\begin{align*}
f(\xx)f\left(w(I)\yy\right)=f\left(w(\xx)\yy\right),\quad(\xx,\yy)\in\VV_+^2,
\end{align*}
where $w(\xx)\yy=w(\xx)\cdot\yy\cdot w^T(\xx)$ (i.e. $g(\xx)=w(\xx)^{-1}$ is a division algorithm). Consistently, we will call $w$ a multiplication algorithm. Such functions were recently considered in \cite{wC2013}. 

Unfortunately we can't answer the question whether there exists division (or equivalently multiplication) algorithm resulting in characterizing other distribution than Riesz or Wishart. Moreover, the simultaneous removal of the assumptions of the invariance of the ``quotient'' and the existence of densities remains a challenge.

This paper is organized as follows. We start in the next section with basic definitions and theorems regarding analysis on symmetric cones. The statement and proof of the main result are given in Section \ref{secLUK}. Section \ref{funeq} is devoted to consideration of $w$-logarithmic Cauchy functions and the Olkin--Baker functional equation. In that section we offer much shorter, simpler and covering more general cones proof of the Olkin--Baker functional equation than given in \cite{BW2002,HaLaZi2008,Kolo2013}. 

\section{Preliminaries}
In this section we give a short introduction to the theory of symmetric cones. For further details we refer to \citet{FaKo1994}. 

A \textit{Euclidean Jordan algebra} is a Euclidean space $\En$ (endowed with scalar product denoted $\scalar{\xx,\yy}$) equipped with a bilinear mapping (product)
\begin{align*}
\En\times\En \ni \left(\xx,\yy\right)\mapsto \xx\yy\in\En
\end{align*}
and a neutral element $\ee$ in $\En$ such that for all $\xx$, $\yy$, $\zz$ in $\En$:
\begin{enumerate}[(i)]
	\item $\xx\yy=\yy\xx$, 
	\item $\xx(\xx^2\yy)=\xx^2(\xx\yy)$,
	\item $\xx\ee=\xx$,
	\item $\scalar{\xx,\yy\zz}=\scalar{\xx\yy,\zz}$.
\end{enumerate}
For $\xx\in\En$ let $\LL(\xx)\colon \En\to\En$ be linear map defined by
\begin{align*}
\LL(\xx)\yy=\xx\yy,
\end{align*}
and define 
\begin{align*}
\PP(\xx)=2\LL^2(\xx)-\LL\left(\xx^2\right).
\end{align*} 
The map $\PP\colon \En\mapsto End(\En)$ is called the \emph{quadratic representation} of $\En$.

An element $\xx$ is said to be \emph{invertible} if there exists an element $\yy$ in $\En$ such that $\LL(\xx)\yy=\ee$. Then $\yy$ is called the \emph{inverse of} $\xx$ and is denoted by $\yy=\xx^{-1}$. Note that the inverse of $\xx$ is unique. It can be shown that $\xx$ is invertible if and only if $\PP(\xx)$ is invertible and in this case $\left(\PP(\xx)\right)^{-1} =\PP\left(\xx^{-1}\right)$.

Euclidean Jordan algebra $\En$ is said to be \emph{simple} if it is not a \mbox{Cartesian} product of two Euclidean Jordan algebras of positive dimensions. Up to linear isomorphism there are only five kinds of Euclidean simple Jordan algebras. Let $\mathbb{K}$ denote either the real numbers $\RR$, the complex ones $\mathbb{C}$, quaternions $\mathbb{H}$ or the octonions $\mathbb{O}$, and write $S_r(\mathbb{K})$ for the space of $r\times r$ Hermitian matrices with entries valued in $\mathbb{K}$, endowed with the Euclidean structure $\scalar{\xx,\yy}=\Trace(\xx\cdot\bar{\yy})$ and with the Jordan product
\begin{align}\label{defL}
\xx\yy=\tfrac{1}{2}(\xx\cdot\yy+\yy\cdot\xx),
\end{align}
where $\xx\cdot\yy$ denotes the ordinary product of matrices and $\bar{\yy}$ is the conjugate of $\yy$. Then $S_r(\RR)$, $r\geq 1$, $S_r(\mathbb{C})$, $r\geq 2$, $S_r(\mathbb{H})$, $r\geq 2$, and the exceptional $S_3(\mathbb{O})$ are the first four kinds of Euclidean simple Jordan algebras. Note that in this case 
\begin{align}\label{defP}
\PP(\yy)\xx=\yy\cdot\xx\cdot\yy.
\end{align}
The fifth kind is the Euclidean space $\RR^{n+1}$, $n\geq 2$, with Jordan product
\begin{align}\label{scL}\begin{split}
\left(x_0,x_1,\dots, x_n\right)\left(y_0,y_1,\dots,y_n\right) =\left(\sum_{i=0}^n x_i y_i,x_0y_1+y_0x_1,\dots,x_0y_n+y_0x_n\right).
\end{split}
\end{align}

To each Euclidean simple Jordan algebra one can attach the set of Jordan squares
\begin{align*}
\bar{\VV}=\left\{\xx^2\colon\xx\in\En \right\}.
\end{align*}
The interior $\VV$ is a symmetric cone.
Moreover $\VV$ is \emph{irreducible}, i.e. it is not the Cartesian product of two convex cones. One can prove that an open convex cone is symmetric and irreducible if and only if it is the cone $\VV$ of some Euclidean simple Jordan algebra. Each simple Jordan algebra corresponds to a symmetric cone, hence there exist up to linear isomorphism also only five kinds of symmetric cones. The cone corresponding to the Euclidean Jordan algebra $\RR^{n+1}$ equipped with Jordan product \eqref{scL} is called the Lorentz cone. 

We denote by $G(\En)$ the subgroup of the linear group $GL(\En)$ of linear automorphisms which preserves $\VV$, and we denote by $G$ the connected component of $G(\En)$ containing the identity.  Recall that if $\En=S_r(\RR)$ and $GL(r,\RR)$ is the group of invertible $r\times r$ matrices, elements of $G(\En)$ are the maps $g\colon\En\to\En$ such that there exists $\ab\in GL(r,\RR)$ with
\begin{align*}
g(\xx)=\ab\cdot\xx\cdot\ab^T.
\end{align*}
We define $K=G\cap O(\En)$, where $O(\En)$ is the orthogonal group of $\En$. It can be shown that 
\begin{align*}
K=\{ k\in G\colon k\ee=\ee \}.
\end{align*} 

A \emph{multiplication algorithm} is a map $\VV\to G\colon \xx\mapsto w(\xx)$ such that $w(\xx)\ee=\xx$ for all $\xx\in\VV$. This concept is consistent with, so called, division algorithm $g$, which was introduced by \citet{OlRu1962} and \citet{CaLe1996}, that is a mapping $\VV\ni\xx\mapsto g(\xx)\in G$ such that $g(\xx)\xx=\ee$ for any $\xx\in\VV$. If $w$ is a multiplication algorithm then $g=w^{-1}$ (that is, $g(\xx)w(\xx)=w(\xx)g(\xx)=Id_\VV$ for any $\xx\in\VV$) is a division algorithm and vice versa, if $g$ is a division algorithm then $w=g^{-1}$ is a multplication algorithm. One of two important examples of multiplication algorithms is the map $w_1(\xx)=\PP\left(\xx^{1/2}\right)$.

We will now introduce a very useful decomposition in $\En$, called \emph{spectral decomposition}. An element $\cb\in\En$ is said to be a \emph{idempotent} if $\cb\cb=\cb\neq 0$. Idempotents $\ab$ and $\bb$ are \emph{orthogonal} if $\ab\bb=0$. Idempotent $\cb$ is \emph{primitive} if $\cb$ is not a sum of two non-null idempotents. A \emph{complete system of primitive orthogonal idempotents} is a set $\left(\cb_1,\dots,\cb_r\right)$ such that
\begin{align*}
\sum_{i=1}^r \cb_i=\ee\quad\mbox{and}\quad\cb_i\cb_j=\delta_{ij}\cb_i\quad\mbox{for } 1\leq i\leq j\leq r.
\end{align*}
The size $r$ of such system is a constant called the \emph{rank} of $\En$. Any element $\xx$ of a Euclidean simple Jordan algebra can be written as $\xx=\sum_{i=1}^r\lambda_i\cb_i$ for some complete system of primitive orthogonal idempotents $\left(\cb_1,\dots,\cb_r\right)$. The real numbers $\lambda_i$, $i=1,\dots,r$ are the \emph{eigenvalues} of $\xx$. One can then define \emph{trace} and \emph{determinant} of $\xx$ by, respectively, $\tr\xx=\sum_{i=1}^r\lambda_i$ and $\det\xx=\prod_{i=1}^r\lambda_i$. An element $\xx\in\En$ belongs to $\VV$ if and only if all its eigenvalues are strictly positive. 

The rank $r$ and $\dim\VV$ of irreducible symmetric cone are connected through relation
\begin{align*}
\dim\VV=r+\frac{d r(r-1)}{2},
\end{align*}
where $d$ is an integer called the \emph{Peirce constant}. 

If $\cb$ is a primitive idempotent of $\En$, the only possible eigenvalues of $\LL(\cb)$ are $0$, $\tfrac{1}{2}$ and $1$. We denote by $\En(\cb,0)$, $\En(\cb,\tfrac{1}{2})$ and $\En(\cb,1)$ the corresponding eigenspaces. The decomposition 
\begin{align*}
\En=\En(\cb,0)\oplus\En(\cb,\tfrac{1}{2})\oplus\En(\cb,1)
\end{align*}
is called the \emph{Peirce decomposition of $\En$ with respect to $\cb$}. Note that $\PP(\cb)$ is the orthogonal projection of $\En$ onto $\En(\cb,1)$.

Fix a complete system of orthogonal idempotents $\left(\cb_i\right)_{i=1}^r$. Then for any $i,j\in\left\{1,2,\dots,r\right\}$ we write
\begin{align*}
\begin{split}
\En_{ii} & =\En(\cb_i,1)=\RR \cb_i, \\
\En_{ij} & = \En\left(\cb_i,\frac{1}{2}\right) \cap \En\left(\cb_j,\frac{1}{2}\right) \mbox{ if }i\neq j.
\end{split}
\end{align*}
It can be proved (see \cite[Theorem IV.2.1]{FaKo1994}) that
\begin{align*}
\En=\bigoplus_{i\leq j}\En_{ij}
\end{align*}
and
\begin{align*}
\begin{split}
\En_{ij}\cdot\En_{ij} & \subset\En_{ii}+\En_{ij}, \\
\En_{ij}\cdot\En_{jk} & \subset\En_{ik},\mbox{ if }i\neq k, \\
\En_{ij}\cdot\En_{kl} & =\{0\},\mbox{ if }\{i,j\}\cap\{k,l\}=\emptyset.
\end{split}\end{align*}
Moreover (\cite[Lemma IV.2.2]{FaKo1994}), if $\xx\in\En_{ij}$, $\yy\in\En_{jk}$, $i\neq k$, then 
\begin{align}\label{xxyy}
\xx^2 & =\tfrac{1}{2}\norm{\xx}^2(\cb_i+\cb_j),\\
\notag\norm{\xx\yy}^2 & =\tfrac{1}{8}\norm{\xx}^2\norm{\yy}^2.
\end{align}
The dimension of $\En_{ij}$ is the Peirce constant $d$ for any $i\neq j$. When $\En$ is $S_r(\KK)$, if $(e_1,\dots,e_r)$ is an orthonormal basis of $\RR^r$, then $\En_{ii}=\RR e_i e_i^T$ and $\En_{ij}=\KK\left(e_i e_j^T+e_je_i^T\right)$ for $i<j$ and $d$ is equal to $dim_{|\RR}\KK$.
	
For $1\leq k \leq r$ let $P_k$ be the orthogonal projection onto $\En^{(k)}=\En(\cb_1+\ldots+\cb_k,1)$, $\det^{(k)}$ the determinant in the subalgebra $\En^{(k)}$, and, for $\xx\in\VV$, $\Delta_k(\xx)=\det^{(k)}(P_k(\xx))$. Then $\Delta_k$ is called the principal minor of order $k$ with respect to the Jordan frame $(\cb_k)_{k=1}^r$. Note that $\Delta_r(\xx)=\det\xx$. For $s=(s_1,\ldots,s_r)\in\RR^r$ and $\xx\in\VV$, we write
\begin{align*}
\Delta_s(\xx)=\Delta_1(\xx)^{s_1-s_2}\Delta_2(\xx)^{s_2-s_3}\ldots\Delta_r(\xx)^{s_r}.
\end{align*}
$\Delta_s$ is called a generalized power function. If $\xx=\sum_{i=1}^r\alpha_i\cb_i$, then $\Delta_s(\xx)=\alpha_1^{s_1}\alpha_2^{s_2}\ldots\alpha_r^{s_r}$.

We will now introduce some basic facts about triangular group. For $\xx$ and $\yy$ in $\VV$, let $\xx\Box\yy$ denote the endomorphism of $\En$ defined by
\begin{align*}
\xx\Box\yy=\LL(\xx\yy)+\LL(\xx)\LL(\yy)-\LL(\yy)\LL(\xx).
\end{align*}
If $\cb$ is an idempotent and $\zz\in\En(\cb,\frac{1}{2})$ we define the \emph{Frobenius transformation $\tau_\cb(\zz)$ in $G$} by
\begin{align*}
\tau_\cb(\zz)=\exp(2\zz\Box\cb).
\end{align*}
Since $2\zz\Box\cb$ is nilpotent of degree $3$ (see \cite[Lemma VI.3.1]{FaKo1994}) we get
\begin{align}\label{taucb}
\tau_{\cb}(\zz)=I+(2\zz\Box\cb)+\frac{1}{2}(2\zz\Box\cb)^2.
\end{align}
Given a Jordan frame $(\cb_i)_{i=1}^r$, the subgroup of $G$, 
\begin{align*}
\TT=\left\{\tau_{\cb_1}(\zz^{(1)})\ldots\tau_{\cb_{r-1}}(\zz^{(r-1)})\PP\left(\sum_{i=1}^r \alpha_i\cb_i\right)\colon \alpha_i>0, \zz^{(j)}\in \bigoplus_{k=j+1}^r\En_{jk}\right\}
\end{align*}
is called the \emph{triangular group corresponding to the Jordan frame $(\cb_i)_{i=1}^r$}. For any $\xx$ in $\VV$ there exists a unique $t_{\xx}$ in $\TT$ such that $\xx=t_{\xx}\ee$, that is, there exist (see \cite[Theorem IV.3.5]{FaKo1994}) elements $\zz^{(j)}\in \bigoplus_{k=j+1}^r \En_{jk}$, $1\leq j\leq r-1$ and positive numbers $\alpha_1, \ldots ,\alpha_r$ such that
\begin{align}\label{triangular}
\xx=\tau_{\cb_1}(\zz^{(1)})\tau_{\cb_2}(\zz^{(2)})\ldots\tau_{\cb_{r-1}}(\zz^{(r-1)})\left(\sum_{k=1}^r \alpha_k \cb_k \right).
\end{align}
Mapping $w_2\colon\VV\to\TT, \xx\mapsto w_2(\xx)=t_{\xx}$ realizes a multiplication algorithm.

For $\En=S_r(\RR)$ we have $\VV=\VV_+$. Let us define for $1\leq i,j\leq r$ matrix $\mu_{ij}=\left(\gamma_{kl}\right)_{1\leq k,l\leq r}$ such that $\gamma_{ij}=1$ and all other entries are equal $0$.
Then for Jordan frame $\left(\cb_i\right)_{i=1}^r$, where $\cb_k=\mu_{kk}$, $k=1,\ldots,r$, we have $\zz_{jk}=(\mu_{jk}+\mu_{kj})\in\En_{jk}$ oraz $\norm{\zz_{jk}}^2=2$, $1\leq j,k\leq r$, $j\neq k$.
if $\zz^{(i)}\in\bigoplus_{j=i+1}^r \En_{ij}$, $i=1,\ldots,r-1$, then there exists $\alpha^{(i)}=(\alpha_{i+1},\ldots,\alpha_r)\in\RR^{r-i}$ such that $\zz^{(i)}=\sum_{j=i+1}^r \alpha_j\zz_{ij}$. Then the Frobenius transformation reads
$$\tau_{\cb_i}(\zz^{(i)})\xx=\mathcal{F}_i(\alpha^{(i)})\cdot\xx\cdot \mathcal{F}_i(\alpha^{(i)})^T,$$
where $\mathcal{F}_i(\alpha^{(i)})$ is so called Frobenius matrix:
\begin{align*}
\mathcal{F}_i(\alpha^{(i)})=I+\sum_{j=i+1}^r \alpha_j \mu_{ji},
\end{align*}
i.e. bellow $i$th one of identity matrix there is a vector $\alpha^{(i)}$, particularly
\begin{align*}
\mathcal{F}_2(\alpha^{(2)})=\begin{pmatrix}
  1    &   0    &   0    & \cdots & 0 \\
  0    &   1    &   0    & \cdots & 0 \\
  0    & \alpha_{3} &   1    & \cdots & 0 \\
\vdots & \vdots & \vdots & \ddots & \vdots \\
  0    & \alpha_{r} &   0    & \cdots & 1
\end{pmatrix}.
\end{align*}

It can be shown (\cite[Proposition VI.3.10]{FaKo1994}) that for each $t\in\TT$, $\xx\in\VV$ and $s\in\RR^r$,
\begin{align}\label{deltast}
\Delta_s(t\xx)=\Delta_s(t\ee)\Delta_s(\xx)
\end{align}
and for any $\zz\in\En(\cb_i,\frac{1}{2})$, $i=1,\ldots,r$,
\begin{align}\label{deltastau}
\Delta_s(\tau_{\cb_i}(\zz)\ee)=1,
\end{align}
if only $\Delta_s$ and $\TT$ are associated with the same Jordan frame $\left(\cb_i\right)_{i=1}^r$.

We will now introduce some necessary basics regarding certain probability distribution on symmetric cones.
Absolutely continuous Riesz distribution $R_{s,\ab}$ on $\VV$ is defined for any $\ab\in\VV$ and $s=(s_1,\ldots,s_r)\in\RR^r$ such that $s_i>(i-1)d/2$, $i=1,\ldots,r$, though its density
\begin{align*}
R_{s,\ab}(d\xx)=\frac{\Delta_s(\ab)}{\Gamma_\VV(s)} \Delta_{s-\dim\VV/r}(\xx)e^{-\scalar{\ab,\xx}}I_\VV(\xx)\,d\xx,\quad\xx\in\VV,
\end{align*}
where $\Delta_s$ is the generalized power function with respect to a Jordan frame $(\cb_i)_{i=1}^r$ and $\Gamma_\VV$ is the Gamma function of the symmetric cone $\VV$. It can be shown that $\Gamma_\VV(s)=(2\pi)^{(\dim\VV-r)/2}\prod_{j=1}^r \Gamma(s_j-(j-1)\tfrac{d}{2})$ (see \cite[VII.1.1.]{FaKo1994}). Riesz distribution was introduced in \citet{HaLa01}.
 
Absolutely continuous Wishart distribution $\gamma_{p,\ab}$ on $\VV$ is a special case of Riesz distribution for $s_1=\ldots=s_r=p$. If $\ab\in\VV$ and $p>\dim\VV/r-1$ it has density
\begin{align*}
\gamma_{p,\ab}(d\xx)=\frac{(\det\ab)^p}{\Gamma_\VV(p)} (\det\xx)^{p-\dim\VV/r}e^{-\scalar{\ab,\xx}}I_\VV(\xx)\,d\xx,\quad\xx\in\VV,
\end{align*}
where $\Gamma_\VV(p):=\Gamma_\VV(p,\ldots,p)$.
Wishart distribution is a generalization of gamma distribution (case $r=1$). 

In generality, Riesz and Wishart distributions does not always have densities, but due to the assumption of existence of densities in Theorem \ref{lukth}, we are not interested in other cases.

\section{Functional equations}\label{funeq}

\subsection{Logarithmic Cauchy functions}
As will be seen, the densities of respective random variables will be given in terms of $w$-logarithmic Cauchy functions, ie. functions $f\colon\VV\to\RR$ that satisfy the following functional equation 
\begin{align}\label{wC}
f(\xx)+f(w(\ee)\yy)=f(w(\xx)\yy),\quad (\xx,\yy)\in\VV^2,
\end{align}
where $w$ is a multiplication algorithm. If $f$ is $w$-logarithmic, then $e^f$ is called $w$-multiplicative. In the following section we will give the form of $w$-logarithmic Cauchy functions for two basic multiplication algorithms, one connected with the quadratic representation 
\begin{align}\label{defw1}
w_1(\xx)=\PP(\xx^{1/2}),
\end{align}
and the other related to a triangular group $\TT$, 
\begin{align}\label{defw2}
w_2(\xx)=t_\xx\in\TT.
\end{align}
Such functions were recently considered without any regularity assumptions in \cite{wC2013}. 

It should be stressed that there exist infinite number of multiplication algorithms. If $\ww$ is a multiplication algorithm, then trivial extensions are given by $\ww^{(k)}(\xx)=\ww(\xx) k$, where $k\in K$ is fixed (Remark \ref{remtriv} explains why this extension is trivial when it comes to multiplicative functions). One may consider also multiplication algorithms of the form $\PP(\xx^\alpha)t_{\xx^{1-2\alpha}}$, which interpolate between the two main examples: $\ww_1$ (which is $\alpha=1/2$) and $\ww_2$ (which is $\alpha=0$). In general any multiplication algorithm may be written in the form $\ww(x)=\PP(\xx^{1/2})k_x$, where $\xx\mapsto k_\xx\in K$.

Functional equation \eqref{wC} for $w_1$ were already considered by \citet{BW2003} for differentiable functions and by \citet{Molnar2006} for continuous functions of real or complex Hermitian 
positive definite matrices of rank greater than $2$. Without any regularity assumptions it was solved on the Lorentz cone by \citet{Wes2007L}. 

Case of $w_2(\xx)=t_\xx\in\TT$ for a triangular group $\TT$, perhaps a bit surprisingly, leads to a different solution. It was indirectly solved for differentiable functions by \citet[Proof of Theorem 3.3]{HaLaZi2008}. 

By \cite[Proposition III.4.3]{FaKo1994}, for any $g$ in the group $G$,
\begin{align*}
\det(g\xx)=(\DDet\,g)^{r/\dim\VV}\det\xx,
\end{align*}
where $\DDet$ denotes the determinant in the space of endomorphisms on $\VV$. Inserting a multiplication algorithm $g=w(\yy)$, $\yy\in\VV$, and $\xx=\ee$ we obtain
\begin{align}\label{zzz}
\DDet\left(w(\yy)\right) =(\det\yy)^{\dim\VV/r}
\end{align}
and hence
\begin{align*}
\det(w(\yy)\xx) =\det\yy\det\xx
\end{align*}
for any $\xx,\yy\in\VV$. This means that $f(\xx)=H(\det\xx)$, where $H$ is generalized logarithmic function, ie. $H(ab)=H(a)+H(b)$ for $a,b>0$, is always a solution to \eqref{wC}, regardless of the choice of multiplication algorithm $w$. If a $w$-logarithmic functions $f$ is additionally $K$-invariant ($f(\xx)=f(k\xx)$ for any $k\in K$), then $H(\det\xx)$ is the only possible solution (Theorem \ref{XXX}).

In \cite{wC2013} the following theorems have been proved. They will be useful in the proof of the main theorems in this paper.
\begin{theorem}[$w_1$-logarithmic Cauchy functional equation]\label{w1th}
Let $f\colon \VV\to\RR$ be a function such that
\begin{align*}
f(\xx)+f(\yy)=f\left(\PP\left(\xx^{1/2}\right)\yy\right),\quad (\xx,\yy)\in\VV^2.
\end{align*}
Then there exists a logarithmic function $H$ such that for any $\xx\in\VV$,
\begin{align*}
f(\xx)=H(\det\xx).
\end{align*}
\end{theorem}
\begin{theorem}[$w_2$-logarithmic Cauchy functional equation]\label{w2th}
Let $f\colon \VV\to \RR$ be a function satisfying
\begin{align*}
f(\xx)+f(\yy)=f(t_{\yy}\xx)
\end{align*}
for any $\xx$ and $\yy$ in the cone $\VV$ of rank $r$, $t_{\yy}\in\TT$, where $\TT$ is the triangular group with respect to the Jordan frame $\left(\cb_i\right)_{i=1}^r$. Then there exist generalized logarithmic functions $H_1,\ldots, H_r$ such that for any $\xx\in\VV$,
\begin{align*}
f(\xx)=\sum_{k=1}^r H_k(\Delta_k(\xx)),
\end{align*}
where $\Delta_k$ is the principal minor of order $k$ with respect to $\left(\cb_i\right)_{i=1}^r$.
\end{theorem}
If we assume in Theorem \ref{w2th} that $f$ is additionally measurable, then functions $H_k$ are measurable. This implies that there exists constants $s_k\in\RR$ such that $H_k(\alpha)=s_k\log\alpha$ and 
$$f(\xx)=\sum_{k=1}^r s_k\log(\Delta_k(\xx))=\log \prod_{k=1}^r \Delta^{s_k}_k(\xx).$$Thus, we obtain the following
\begin{remark}\label{w2rem}
If we impose on $f$  in Theorem \ref{w2th} some mild conditions (eg. measurability), then there exists $s\in\RR^r$ such that for any $\xx\in\VV$,
$$f(\xx)=\log\Delta_s(\xx).$$
\end{remark}
\begin{theorem}\label{XXX}
Let $f\colon\VV\to\RR$ be a function satisfying \eqref{wC}. Assume additionally that $f$ is $K$-invariant, ie. $f(k\xx)=f(\xx)$ for any $k\in K$ and $\xx\in\VV$. Then there exists a logarithmic function $H$ such that for any $\xx\in\VV$,
$$f(\xx)=H(\det\xx).$$
\end{theorem}
\begin{lemma}[$w$-logarithmic Pexider functional equation]\label{lem1}
Assume that $a$, $b$, $c$ defined on the cone $\Omega$ satisfy following functional equation
\begin{align*}
a(\xx)+b(\yy)=c(w(\xx)\yy),\quad (\xx,\yy)\in\VV^2.
\end{align*}
Then there exist $w$-logarithmic function $f$ and real constants $a_0, b_0$ such that for any $\xx\in\VV$,
\begin{align*}
a(\xx) & =f(\xx)+a_0,\\
b(\xx) & =f(w(\ee)\xx)+b_0,\\
c(\xx) & =f(\xx)+a_0+b_0.
\end{align*}
\end{lemma}

\subsection{The Olkin--Baker functional equation}
In the following section we deal with the Olkin-Baker functional equation on irreducible symmetric cones, which is related to the Lukacs independence condition (see proof of the Theorem \ref{lukth}).

Henceforth we will assume that multiplication algorithm $\ww$ additionally is homogeneous of degree $1$, that is $\ww(s\xx)=s\ww(\xx)$ for any $s>0$ and $\xx\in\VV$. It is easy to create a multiplication algorithm without this property, for example:
\begin{align*}
\ww(\xx)=\begin{cases}
 \ww_1(\xx), &\mbox{ if } \det\xx>1, \\
 \ww_2(\xx), &\mbox{ if } \det\xx\leq 1.
 \end{cases}
\end{align*}

The problem of solving
\begin{align}\label{OB}
f(x)g(y)=p(x+y)q(x/y),\quad(x,y)\in(0,\infty)^2
\end{align}
for unknown positive functions $f$, $g$, $p$ and $q$ was first posed in \cite{Olkin3}. Note that in one dimensional case it does not matter whether one considers $q(x/y)$ or $q(x/(x+y))$ on the right hand side of \eqref{OB}. Its general solution was given in \cite{Baker1976}, and later analyzed in \cite{Lajko1979} using a different approach. Recently, in \cite{Mesz2010} and \cite{LajMes12} the equation \eqref{OB} was solved assuming that it is satisfied almost everywhere on $(0,\infty)^2$ for measurable functions which are non-negative on its domain or positive on some sets of positive Lebesgue measure, respectively. Finally, a new derivation of solution to \eqref{OB}, when the equation holds almost everywhere on $(0,\infty)^2$ and no regularity assumptions on unknown positive functions are imposed, was given in \cite{WES4}. The following theorem is concerned with an adaptation of \eqref{OB} (after taking logarithm) to the symmetric cone case.

\begin{theorem}[Olkin--Baker functional equation on symmetric cones] \label{T1}
Let $a$, $b$, $c$ and $d$ be real continuous functions on an irreducible symmetric cone $\VV$ of rank $r$. Assume
\begin{align}\label{main}
a(\xx)+b(\yy)=c(\xx+\yy)+d\left(\gw\left(\xx+\yy\right)\xx\right),\qquad (\xx,\yy)\in \VV^2,
\end{align}
where $\gw^{-1}=\ww$ is a homogeneous of degree $1$ multiplication algorithm. Then there exist constants $C_i\in\RR$, $i=1,\ldots,4$, $\Lambda\in\En$ such that for any $\xx\in\VV$ and $\ub\in\DD=\left\{\xx\in\VV\colon\ee-\xx\in\VV\right\}$,
\begin{align*}
a(\xx)&=\scalar{\Lambda,\xx}+e(\xx)+C_1,\\
b(\xx)&=\scalar{\Lambda,\xx}+f(\xx)+C_2,\\
c(\xx)&=\scalar{\Lambda,\xx}+e(\xx)+f(\xx)+C_3,\\
d(\ub)&=e(w(\ee)\ub)+f(\ee-w(\ee)\ub)+C_4,
\end{align*}
where $e$ and $f$ are continuous $\ww$-logarithmic Cauchy functions and $C_1+C_2=C_3+C_4$.
\end{theorem}

We will need following simple lemma. For the elementary proof we refer to \cite[Lemma 3.2]{Kolo2013}.
\begin{lemma}[Additive Pexider functional equation on symmetric cones]\label{lemma1}
Let $a$, $b$ and $c$ be measurable functions on a symmetric cone $\VV$ satisfying 
\begin{align}\label{lempex}
a(\xx)+b(\yy)=c(\xx+\yy),\qquad (\xx,\yy)\in \VV^2.
\end{align}
Then there exist constants $\alpha, \beta\in\RR$ and $\lambda\in\En$ such that for all $\xx\in\VV$,
\begin{align}\label{abcdef}\begin{split}
a(\xx)&=\scalar{\lambda,\xx}+\alpha, \\
b(\xx)&=\scalar{\lambda,\xx}+\beta, \\
c(\xx)&=\scalar{\lambda,\xx}+\alpha +\beta.
\end{split}\end{align} 
\end{lemma}

Now we can come back and give a new proof the Olkin--Baker functional equation.
\begin{proof} [ Prof of Theorem~\ref{T1}]
In the first part of the proof we adapt the argument given in \cite{WES4}, where the analogous result on $(0,\infty)$ was analyzed, to the symmetric cone setting.

For any $s>0$ and $(\xx,\yy)\in \VV^2$ we get
\begin{align}\label{mainr}
a(s\xx)+b(s\yy)=c(s(\xx+\yy))+d\left(g(s\xx+s\yy)s\xx\right).
\end{align}
Since $w$ is homogeneous of degree $1$ we have $g(s\xx)=\tfrac{1}{s}g(\xx)$ and so $g(s\xx+s\yy)s\xx=g(\xx+\yy)\xx$ for any $s>0$.
Subtracting now \eqref{main} from \eqref{mainr} for any $s>0$ we arrive at the additive Pexider equation on symmetric cone $\VV$,
\begin{align*}
a_s(\xx)+b_s(\yy)=c_s(\xx+\yy),\qquad (\xx,\yy)\in \VV^2,
\end{align*}
where  $a_s$, $b_s$ and $c_s$ are functions defined by $a_s(\xx):=a(s\xx)-a(\xx)$, $b_s(\xx):=b(s\xx)-b(\xx)$ and $c_s(\xx):=c(s\xx)-c(\xx)$.

Due to continuity of $a$, $b$ and $c$ and Lemma \ref{lemma1} it follows that for any $s>0$ there exist constants $\lambda(s)\in\En$, $\alpha(s)\in\RR$ and $\beta(s)\in\RR$ such that for any $\xx\in\VV$, 
\begin{align*}
a_s(\xx) &= \scalar{\lambda(s),\xx}+\alpha(s),\\
b_s(\xx) &= \scalar{\lambda(s),\xx}+\beta(s),\\
c_s(\xx) &= \scalar{\lambda(s),\xx}+\alpha(s)+\beta(s).
\end{align*}
By the definition of $a_s$ and the above observation it follows that for any $(s,t)\in(0,\infty)^2$ and $\zz\in\VV$ 
\begin{align*}
a_{st}(\zz)=a_t(s\zz)+a_s(\zz).
\end{align*}
Hence, 
\begin{align}\label{naQ}
\scalar{\lambda(st),\zz}+\alpha(st)=\scalar{\lambda(t),s\zz}+\alpha(t)+\scalar{\lambda(s),\zz}+\alpha(s).
\end{align}
Since \eqref{naQ} holds for any $\zz\in\VV$ we see that $\alpha(st)=\alpha(s)+\alpha(t)$ for all $(s,t)\in(0,\infty)^2$. That is $\alpha(s)=k_1\log\,s$ for $s\in(0,\infty)$, where $k_1$ is a real constant. 

On the other hand 
\begin{align}\label{trr}
\scalar{\lambda(st),\zz}=\scalar{\lambda(s),\zz}+\scalar{\lambda(t),s\zz}=\scalar{\lambda(t),\zz}+\scalar{\lambda(s),t\zz}
\end{align}
since one can interchange $s$ and $t$ on the left hand side.
Putting $s=2$ and denoting $\Lambda=\lambda(2)$ we obtain
 \begin{align*}
 \scalar{\lambda(t),\zz}=\scalar{\Lambda,\zz}(t-1)
 \end{align*}
 for $t>0$ and $\zz\in \VV$. 
 It then follows that for all $s\in(0,\infty)$ and $\zz\in \VV$, 
 \begin{align}\label{asz}
a_s(\zz)=a(s\zz)-a(\zz)=\scalar{\Lambda,\zz}(s-1)+k_1\log\,s.
\end{align}

Let us define function $\bar{a}$ by formula
\begin{align*}
\bar{a}(\xx)=a(\xx)-\scalar{\Lambda,\xx}.
\end{align*}
From \eqref{asz} we get
\begin{align}\label{prope}
\bar{a}(s\xx)=\bar{a}(\xx)+k_1\log s
\end{align}
for $s>0$ and $\xx\in \VV$. \\
Analogous considerations for function $b_s$ gives existence of constant $k_2$ such that $\bar{b}(s\xx)=\bar{b}(\xx)+k_2\log s$, where
\begin{align*}
\bar{b}(\xx)=b(\xx)-\scalar{\Lambda,\xx},
\end{align*}
hence $\bar{c}(s\xx)=\bar{c}(\xx)+(k_1+k_2)\log s$  and
\begin{align*}
\bar{c}(\xx)=c(\xx)-\scalar{\Lambda,\xx}
\end{align*}
for any $s>0$ and $\xx\in\VV$. \\
Functions $\bar{a}$, $\bar{b}$, $\bar{c}$ and $d$ satisfy original Olkin-Baker functional equation:
\begin{align}\label{efgh}
\bar{a}(\xx)+\bar{b}(\yy)=\bar{c}(\xx+\yy)+d\left(\gw\left(\xx+\yy\right)\xx\right),\quad(\xx,\yy)\in\VV^2.
\end{align}
Taking $\xx=\yy=\vb\in\VV$ in \eqref{efgh}, we arrive at
\begin{align}\label{limit2}
\bar{a}(\vb)+\bar{b}(\vb)=\bar{c}(2\vb)+d(\tfrac{1}{2}\ee)=\bar{c}(\vb)+(k_1+k_2)\log 2+d(\tfrac{1}{2}\ee).
\end{align}
Insert $\xx=\alpha \ww(\vb)\ub$ and $\yy=\ww(\vb)(\ee-\alpha \ub)$ into \eqref{efgh} for $0<\alpha<1$ and $(\ub,\vb)\in(\DD,\VV)$.
Using \eqref{prope} we obtain
\begin{align*}
\bar{a}(\ww(\vb)\ub)+\bar{b}(\ww(\vb)(\ee-\alpha \ub))=\bar{c}(\vb)+d\left(\alpha\ub\right)-k_1\log\alpha,\quad(\ub,\vb)\in(\DD,\VV).
\end{align*}
Let us observe, that due to continuity of $\bar{b}$ on $\VV$ and $\lim_{\alpha\to0}\left\{\ww(\vb)(\ee-\alpha \ub)\right\}=\ww(\vb)\ee=\vb\in\VV$ (convergence in the norm generated by scalar product $\scalar{\cdot,\cdot}$), limit as $\alpha\to 0$ of the left hand side of the above equality exists. Hence, the limit of the right hand side also exists and
\begin{align}\label{limit1}
\bar{a}(\ww(\vb)\ub)+\bar{b}(\vb)=\bar{c}(\vb)+\lim_{\alpha\to 0}\left\{d(\alpha\ub) -k_1\log\alpha \right\},\quad (\ub,\vb)\in(\DD,\VV).
\end{align}
Subtracting \eqref{limit1} from \eqref{limit2} we have
\begin{align}\label{e}
\bar{a}(\ww(\vb)\ub)=\bar{a}(\vb)+g(\ub)
\end{align}
for $\ub\in\DD,\vb\in\VV$, where $g(\ub)=\lim_{\alpha\to 0}\left\{d(\alpha\ub) -k_1\log\alpha\right\}-(k_1+k_2)\log 2-d(\tfrac{1}{2}\ee)$.  
Due to the property \eqref{prope} equation \eqref{e} holds for any $\ub\in\VV$, so we arrive at the $\ww$-logarithmic Pexider equation. Lemma \ref{lem1} implies that there exist $\ww$-logarithmic function $e$ such that
\begin{align*}
\bar{a}(\xx)=e(\xx)+C_1
\end{align*}
for any $\xx\in\VV$ and a constant $C_1\in\RR$. Function $e$ is continuous, because $\bar{a}$ is continuous. Coming back to the definition of $\bar{a}$, we obtain
\begin{align*}
a(\xx)=\scalar{\Lambda,\xx}+e(\xx)+C_1, \quad \xx\in\VV.
\end{align*}

Analogously for function $b$, considering equation \eqref{efgh} for $\xx=\ww(\vb)(\ee-\alpha \ub)$ and $\yy=\alpha \ww(\vb)\ub$ after passing to the limit as $\alpha\to0$, we show that there exist continuous $\ww$-logarithmic function $f$ such that
\begin{align*}
b(\xx)=\scalar{\Lambda,\xx}+f(\xx)+C_2, \quad \xx\in\VV
\end{align*}
for a constant $C_2\in\RR$. The form of $c$ follows from \eqref{limit2}. Taking $\xx=w(\ee)\ub$ and $\yy=\ee-w(\ee)\ub$  in \eqref{efgh} for $\ub\in\DD$, we obtain the form of $d$.
\end{proof}

\section{The Lukacs-Olkin-Rubin theorem without invariance of the quotient}\label{secLUK}
In the following section we prove the density version of Lukacs-Olkin-Rubin theorem for any multiplication algorithm $w$ satisfying 
\begin{enumerate}[(i)]
\item $\ww(s\xx)=s\ww(\xx)$ for $s>0$ and $\xx\in\VV$,
\item differentiability of mapping $\VV\ni\xx\mapsto\ww(\xx)\in G$.
\end{enumerate}
We assume (ii) to ensure that Jacobian of the considered transformation exists. We start with the direct result, showing that the considered measures have desired property. The converse result is given in Theorem \ref{lukth}. For every generalized multiplication $\ww$, the family of these $\ww$-Wishart measures (as defined in \eqref{fxy}) contains the Wishart laws. For $\ww=\ww_1$, there are no other distributions, while the $\ww_2$-Wishart measures consist of the Riesz distributions. It is an open question whether there is a generalized multiplication $\ww$ that leads to other probability measures in this family.

\begin{theorem}\label{impl}
Let $\ww$ be a multiplication algorithm satisfying condition (ii) and define $\gw=\ww^{-1}$.
Suppose that $X$ and $Y$ are independent random variables with densities given by 
\begin{align}\begin{split}\label{fxy}
f_X(\xx) =C_X e(\xx)\exp\scalar{\Lambda,\xx}I_\VV(\xx),\\
f_Y(\xx) =C_Y f(\xx)\exp\scalar{\Lambda,\xx}I_\VV(\xx),
\end{split}\end{align}
where $e$ and $f$ are $\ww$-multiplicative functions, $\Lambda\in\En$ and $\En$ is the Euclidean Jordan algebra associated with the irreducible symmetric cone $\VV$.\\
Then vector $(U,V)=\left(\gw(X+Y)X,X+Y\right)$ have independent components.
\end{theorem}
Note that if $\ww(\xx)=\ww_1(\xx)=\PP(\xx^{1/2})$, then there exist positive constants $\kappa_X$ and $\kappa_Y$ such that $e(\xx)=(\det\xx)^{\kappa_X-\dim\VV/r}$ and $f(\xx)=(\det\xx)^{\kappa_Y-\dim\VV/r}$. In this case $-\Lambda=:\ab\in\VV$ and $(X,Y)\sim \gamma_{\kappa_X,\ab}\otimes \gamma_{\kappa_Y,\ab}$.
Similarly, if $\ww(\xx)=\ww_2(\xx)=t_\xx$, $X$ and $Y$ follow Riesz distributions with the same scale parameter $-\Lambda\in\VV$. In general we do not know whether $\ab=-\Lambda$ should always belong to $\VV$.
\begin{proof}
Let $\psi\colon \VV\times\VV\to\DD\times\VV$ be a mapping defined through
\begin{align*}
\psi(\xx,\yy)=\left(\gw(\xx+\yy)\xx,\xx+\yy\right)=(\ub,\vb).
\end{align*}
Then $(U,V)=\psi(X,Y)$. The inverse mapping $\psi^{-1}\colon \DD\times\VV\to\VV\times\VV$ is given by
\begin{align*}
(\xx,\yy)=\psi^{-1}(\ub,\vb)=\left(\ww(\vb)\ub,\ww(\vb)(\ee-\ub)\right),
\end{align*}
hence $\psi$ is a bijection. We are looking for the Jacobian of the map $\psi^{-1}$, that is, the determinant of the linear map
\begin{align*}
\begin{pmatrix}
d\ub\\
d\vb
\end{pmatrix}
\mapsto
\begin{pmatrix}
d\xx \\
d\yy 
\end{pmatrix}
=
\begin{pmatrix}
d\xx/d\ub & d\xx/d\vb \\
d\yy/d\ub & d\yy/d\vb
\end{pmatrix}
\begin{pmatrix}
d\ub\\
d\vb
\end{pmatrix}.
\end{align*}
We have
\begin{align*}
J=\left| 
\begin{array}{cc}
\ww(\vb) & d\xx/d\vb \\
-\ww(\vb) & Id_\VV-d\xx/d\vb
\end{array}
\right| 
=
\left|
\begin{array}{cc}
\ww(\vb) & d\xx/d\vb \\
0 & Id_\VV
\end{array}
\right| = \DDet(\ww(\vb)).
\end{align*}
where $\DDet$ denotes the determinant in the space of endomorphisms on $\VV$. By \eqref{zzz} we get
\begin{align*}
\DDet\left(\ww\left(\vb\right)\right)=(\det\vb)^{\dim\VV/r}.
\end{align*}
Now we can find the joint density of $(U,V)$. Since $(X,Y)$ have independent components, we obtain
\begin{align}\label{kukuku}
f_{(U,V)}(\ub,\vb)=(\det\vb)^{\dim\VV/r}f_X(\ww(\vb)\ub)f_Y(\ww(\vb)(\ee-\ub))
\end{align}
We assumed \eqref{fxy}, thus there exist $\Lambda\in\En$, $C_X, C_Y \in\RR$ and $\ww$-multiplicative functions $e$, $f$ such that
\begin{align*}
f_{(U,V)}(\ub,\vb)= & (\det\vb)^{\dim\VV/r}f_X(\ww(\vb)\ub)f_Y(\ww(\vb)(\ee-\ub))  \\
	= &C_1C_2\, (\det\vb)^{\dim\VV/r} e(\ww(\vb)\ub) f(\ww(\vb)(\ee-\ub)) e^{\scalar{\Lambda,\vb}}I_\VV(\ww(\vb)\ub)I_\VV(\ww(\vb)(\ee-\ub)) \\
	= &C_1C_2\, (\det\vb)^{\dim\VV/r} e(\vb)f(\vb) e^{\scalar{\Lambda,\vb}}I_\VV(\vb) \,\, e(\ww(\ee)\ub) f(\ww(\ee)(\ee-\ub)) I_\DD(\ub), \\
	= & f_U(\ub) \, f_V(\vb),
\end{align*}
what completes the proof.
\end{proof}

To prove the characterization of given measures, we need to show that the inverse implication is also valid.
The following theorem generalizes results obtained in \cite{BW2002, HaLaZi2008, Kolo2013}. We consider quotient $U$ for any multiplication algorithm $\ww$ satisfying conditions (i) and (ii) given at the beginning of this section (note that multiplication algorithms $\ww_1$ and $\ww_2$ defined in \eqref{defw1} and \eqref{defw2}, respectively, satisfy both of these conditions). Respective densities are then expressed in terms of $\ww$-multiplicative Cauchy functions. 

\begin{theorem}[The Lukacs-Olkin-Rubin theorem with densities on symmetric cones]\label{lukth}
Let $X$ and $Y$ be independent rv's valued in irreducible symmetric cone $\VV$ with strictly positive and continuous densities. Set $V=X+Y$ and $U=\gw\left(X+Y\right)X$ for any multiplication algorithm $\ww=\gw^{-1}$ satisfying conditions (i) and (ii). If $U$ and $V$ are independent then there exist $\Lambda\in\En$ and $\ww$-multiplicative functions $e$, $f$ such that \eqref{fxy} holds.

In particular,
\begin{enumerate}
	\item if $\gw(\xx)=\gw_1(\xx)=\PP(\xx^{-1/2})$, then there exist constants $p_i>\dim\VV/r-1$, $i=1,2$, and $\ab\in\VV$ such that $X\sim \gamma_{p_1,\ab}$ and $Y\sim \gamma_{p_2,\ab}$,
	\item if $\gw(\xx)=\gw_2(\xx)=t_{\xx}^{-1}$ , then there exist constants $s_i=(s_{i,j})_{j=1}^r$, $s_{i,j}>(j-1)d/2$, $i=1,2$, and $\ab\in\VV$ such that $X\sim R_{s_1,\ab}$ and $Y\sim R_{s_2,\ab}$.	
\end{enumerate}
\end{theorem}

\begin{proof} 
We start from \eqref{kukuku}. Since $(U,V)$ is assumed to have independent components, the following identity holds almost everywhere with respect to Lebesgue measure: 
\begin{align}\label{kukuku2}
(\det(\xx+\yy))^{\dim\VV/r}f_X(\xx)f_Y(\yy)=f_U\left(\gw\left(\xx+\yy\right) \xx\right)f_V(\xx+\yy),
\end{align}
where $f_X$,$f_Y$,$f_U$ and $f_V$ denote densities of $X$, $Y$, $U$ and $V$, respectively.

Since the respective densities are assumed to be continuous, the above equation holds for every $\xx,\yy\in\VV$. Taking logarithms of both sides of the above equation (it is permitted since $f_X, f_Y>0$ on $\VV$) we get 
\begin{align}\label{lukacs}
a(\xx)+b(\yy)=c(\xx+\yy)+d\left(\gw\left(\xx+\yy\right)\xx\right),
\end{align}
where
\begin{align*}
a(\xx)&=\log\, f_X(\xx),\\
b(\xx)&=\log\, f_Y(\xx),\\
c(\xx)&=\log\, f_V(\xx)-\tfrac{\dim\VV}{r}\log\det(\xx),\\
d(\ub)&=\log\, f_U(\ub),
\end{align*}
for $\xx\in\VV$ and $\ub\in\DD$.

The first part of the conclusion follows now directly from Theorem \ref{T1}. Thus there exist constants $\Lambda\in\En$, $C_i\in\RR$, $i\in\{1,2\}$ and $\ww$-logarithmic functions $e$ and $f$ such that 
 \begin{align*}
f_X(\xx) & =e^{a(\xx)}=e^{C_1}e(\xx)e^{\scalar{\Lambda,\xx}},\\
f_Y(\xx) & =e^{b(\xx)}=e^{C_2}f(\xx)e^{\scalar{\Lambda,\xx}},
\end{align*}
for any $\xx\in\VV$.

Let us observe that if $\ww(\xx)=\ww_1(\xx)=\PP(\xx^{1/2})$, then for Theorem \ref{w1th} there exist constants $\kappa_i$, $i=1,2,$ such that $e(\xx)=(\det\xx)^{\kappa_1}$ and $f(\xx)=(\det\xx)^{\kappa_2}$. 
Since $f_X$ and $f_Y$ are densities it follows that $\ab=-\Lambda\in\VV$, $k_i=p_i-(\dim\VV)/r>-1$ and $e^{C_i}=(\det(\ab))^{p_i}/\Gamma_\VV(p_i)$, $i=1,2$.

Analogously, if $\ww(\xx)=\ww_2(\xx)=t_\xx$ then Theorem \ref{w2th} and Remark \ref{w2rem} imply that there exist constants $s_i=(s_{i,j})_{j=1}^r$, $s_{i,j}>(j-1)d/2$, $i=1,2$, and $\ab=-\Lambda\in\VV$ such that $X\sim R_{s_1,\ab}$ i $Y\sim R_{s_2,\ab}$.	
\end{proof}


\begin{remark}\label{remtriv}
Fix $k\in K$ and consider $w^{(k)}(\xx)=w(\xx) k$. The $w^{(k)}$-multiplicative function $f$ satisfies equation
\begin{align*}
f(\xx)f(w(\ee)k\yy)=f(w(\xx) k \yy).
\end{align*}
Substituting $\yy\mapsto k^{-1} \yy\in\VV$ we obtain
\begin{align*}
f(\xx)f(w(\ee)\yy)=f(w(\xx)\yy),
\end{align*}
that is $w^{(k)}$-multiplicative functions are the same as $w$-multiplicative functions. This leads to the rather unsurprising observation that if we consider Theorem \ref{lukth} with $w(\xx)=\PP(\xx^{1/2})k$ or $w(\xx)=t_\xx k$, regardless of $k\in K$, we will characterize the same distributions as in points $(1)$ and $(2)$ of Theorem \ref{lukth}.
\end{remark}

With Theorem \ref{lukth} one can easily re-prove original Lukacs-Olkin-Rubin theorem (version of \cite{OlRu1964} and \cite{CaLe1996}), when the distribution of $U$ is invariant under a group of automorphisms:
\begin{remark}
Let us additionally assume in Theorem \ref{lukth}, that the quotient $U$ has distribution which is invariant under a group of automorphisms, that is $kU\stackrel{d}{=}U$ for any $k\in K$. From the proof of Theorem \ref{impl} it follows that there exist continuous $\ww$-multiplicative functions $e$ and $f$ and constant $C$ such that for $\ub\in\DD$,
$$ f_U(\ub)=C e(\ww(\ee)\ub) f(\ee-\ww(\ee)\ub).$$
The distribution of $U$ is invariant under $K$, thus density $f_U$ is a $K$-invariant function, that is $f_U(\ub)=f_U(k\ub)$ for any $k\in K$. Note that $w(\ee)\in K$, thus
\begin{align}\label{KKK}
e(\ub)f(\ee-\ub)=e(k\ub)f(\ee-k\ub),\quad (k,\ub)\in K\times\DD.
\end{align}
We will show that both functions $e$ and $f$ are $K$-invariant. Recall that $e(\xx)\,\,e(\ww(\ee)\yy)=e(\ww(\xx)\yy)$, therefore after taking $\yy=\alpha\ee$ we obtain
$e(\alpha\xx)=e(\xx)e(\alpha\ee)$ for any $\alpha>0$ and $\xx\in\VV$. Inserting $\ub=\alpha\vb$ into \eqref{KKK} we arrive at
\begin{align*}
e(\vb) e(\alpha\ee) f(\ee-\alpha\vb) 
=e(\alpha\vb) f(\ee-\alpha\vb)=e(\alpha k\vb) f(\ee-\alpha k\vb) 
= e(k\vb) e(\alpha\ee) f(\ee-k\alpha\vb).
\end{align*}
Thus $e(\vb) f(\ee-\alpha\vb)=e(k\vb) f(\ee-k\alpha\vb)$ for any $\alpha\in(0,1]$, $\vb\in\DD$. Since $f(\ee)=1$ and $f$ is continuous on $\VV$, by passing to the limit as $\alpha\to0$ we get that $e$ is $K$-invariant and so is $f$. 
By Theorem \ref{XXX} and continuity of $e$ and $f$ we get that there exist constants $\kappa_1$, $\kappa_2$ such that $e(\xx)=(\det\xx)^{\kappa_1}$ and $f(\xx)=(\det\xx)^{\kappa_2}$, hence $X$ and $Y$ have Wishart distributions.
\end{remark}

\subsection*{Acknowledgement} The author thanks J. Weso{\l}owski for helpful comments and discussions. This research was partially supported by NCN grant No. 2012/05/B/ST1/00554.

\bibliographystyle{plainnat}

\bibliography{Bibl}

\def\polhk#1{\setbox0=\hbox{#1}{\ooalign{\hidewidth
  \lower1.5ex\hbox{`}\hidewidth\crcr\unhbox0}}}
\begin{thebibliography}{23}
\providecommand{\natexlab}[1]{#1}
\providecommand{\url}[1]{\texttt{#1}}
\expandafter\ifx\csname urlstyle\endcsname\relax
  \providecommand{\doi}[1]{doi: #1}\else
  \providecommand{\doi}{doi: \begingroup \urlstyle{rm}\Url}\fi

\bibitem[Baker(1976)]{Baker1976}
J.~A. Baker.
\newblock On the functional equation {$f(x)g(y)=p(x+y)q(x/y)$}.
\newblock \emph{Aequationes Math.}, 14\penalty0 (3):\penalty0 493--506, 1976.

\bibitem[Bobecka and Weso{\l}owski(2002)]{BW2002}
K.~Bobecka and J.~Weso{\l}owski.
\newblock The {L}ukacs-{O}lkin-{R}ubin theorem without invariance of the
  ``quotient''.
\newblock \emph{Studia Math.}, 152\penalty0 (2):\penalty0 147--160, 2002.

\bibitem[Bobecka and Weso{\l}owski(2003)]{BW2003}
K.~Bobecka and J.~Weso{\l}owski.
\newblock Multiplicative {C}auchy functional equation in the cone of
  positive-definite symmetric matrices.
\newblock \emph{Ann. Polon. Math.}, 82\penalty0 (1):\penalty0 1--7, 2003.

\bibitem[Boutouria(2005)]{Bo2005}
I.~Boutouria.
\newblock Characterization of the {W}ishart distributions on homogeneous cones.
\newblock \emph{C. R. Math. Acad. Sci. Paris}, 341\penalty0 (1):\penalty0
  43--48, 2005.

\bibitem[Boutouria(2009)]{Bout2009}
I.~Boutouria.
\newblock Characterization of the {W}ishart distribution on homogeneous cones
  in the {B}obecka and {W}esolowski way.
\newblock \emph{Comm. Statist. Theory Methods}, 38\penalty0 (13-15):\penalty0
  2552--2566, 2009.

\bibitem[Boutouria et~al.(2011)Boutouria, Hassairi, and Massam]{BoHaMa2011}
I.~Boutouria, A.~Hassairi, and H.~Massam.
\newblock Extension of the {O}lkin and {R}ubin characterization to the
  {W}ishart distribution on homogeneous cones.
\newblock \emph{Infin. Dimens. Anal. Quantum Probab. Relat. Top.}, 14\penalty0
  (4):\penalty0 591--611, 2011.

\bibitem[Casalis and Letac(1996)]{CaLe1996}
M.~Casalis and G.~Letac.
\newblock The {L}ukacs-{O}lkin-{R}ubin characterization of {W}ishart
  distributions on symmetric cones.
\newblock \emph{Ann. Statist.}, 24\penalty0 (2):\penalty0 763--786, 1996.

\bibitem[Faraut and Kor{\'a}nyi(1994)]{FaKo1994}
J.~Faraut and A.~Kor{\'a}nyi.
\newblock \emph{Analysis on symmetric cones}.
\newblock Oxford Mathematical Monographs. The Clarendon Press Oxford University
  Press, New York, 1994.
\newblock Oxford Science Publications.

\bibitem[Ger et~al.(2013)Ger, Misiewicz, and Weso{\l}owski]{WES4}
R.~Ger, J.~Misiewicz, and J.~Weso{\l}owski.
\newblock The {L}ukacs theorem and the {O}lkin-{B}aker equation.
\newblock \emph{J. Math. Anal. Appl.}, 399\penalty0 (2):\penalty0 599--607,
  2013.

\bibitem[Hassairi and Lajmi(2001)]{HaLa01}
A.~Hassairi and S.~Lajmi.
\newblock Riesz exponential families on symmetric cones.
\newblock \emph{J. Theoret. Probab.}, 14\penalty0 (4):\penalty0 927--948, 2001.
\newblock ISSN 0894-9840.

\bibitem[Hassairi et~al.(2008)Hassairi, Lajmi, and Zine]{HaLaZi2008}
A.~Hassairi, S.~Lajmi, and R.~Zine.
\newblock A characterization of the {R}iesz probability distribution.
\newblock \emph{J. Theoret. Probab.}, 21\penalty0 (4):\penalty0 773--790, 2008.

\bibitem[Ko{\l}odziejek(2010)]{Kolodzeng}
B.~Ko{\l}odziejek.
\newblock The {W}ishart distribution on the {L}orentz cone.
\newblock Master's thesis, Fac. Math. Infor. Sci., Warsaw Univ. Tech., 2010.
\newblock - in Polish.

\bibitem[Ko{\l}odziejek(2013)]{Kolo2013}
B.~Ko{\l}odziejek.
\newblock The {L}ukacs-{O}lkin-{R}ubin theorem on symmetric cones through
  {G}leason's theorem.
\newblock \emph{Studia Math.}, 217:\penalty0 1--17, 2013.

\bibitem[Ko{\l}odziejek(2014)]{wC2013}
B.~Ko{\l}odziejek.
\newblock Multiplicative {C}auchy functional equation on symmetric cones.
\newblock \emph{To appear in Aequationes Math.}, pages 1--20, 2014.

\bibitem[Lajk{\'o}(1979)]{Lajko1979}
K.~Lajk{\'o}.
\newblock Remark to a paper: ``{O}n the functional equation
  {$f(x)g(y)=p(x+y)q(x/y)$}'' by {J}. {A}. {B}aker.
\newblock \emph{Aequationes Math.}, 19\penalty0 (2-3):\penalty0 227--231, 1979.

\bibitem[Lajk{\'o} and M{\'e}sz{\'a}ros(2012)]{LajMes12}
K.~Lajk{\'o} and F.~M{\'e}sz{\'a}ros.
\newblock Multiplicative type functional equations arising from
  characterization problems.
\newblock \emph{Aequationes Math.}, 83\penalty0 (3):\penalty0 199--208, 2012.
\newblock ISSN 0001-9054.

\bibitem[Lukacs(1955)]{Lukacs1955}
E.~Lukacs.
\newblock A characterization of the gamma distribution.
\newblock \emph{Ann. Math. Statist.}, 26:\penalty0 319--324, 1955.

\bibitem[M{\'e}sz{\'a}ros(2010)]{Mesz2010}
F.~M{\'e}sz{\'a}ros.
\newblock A functional equation and its application to the characterization of
  gamma distributions.
\newblock \emph{Aequationes Math.}, 79\penalty0 (1-2):\penalty0 53--59, 2010.

\bibitem[Moln{\'a}r(2006)]{Molnar2006}
L.~Moln{\'a}r.
\newblock A remark on the {K}ochen-{S}pecker theorem and some characterizations
  of the determinant on sets of {H}ermitian matrices.
\newblock \emph{Proc. Amer. Math. Soc.}, 134\penalty0 (10):\penalty0
  2839--2848, 2006.

\bibitem[Olkin(1975)]{Olkin3}
I.~Olkin.
\newblock Problem (p128).
\newblock \emph{Aequationes Math.}, 12:\penalty0 290--292, 1975.

\bibitem[Olkin and Rubin(1962)]{OlRu1962}
I.~Olkin and H.~Rubin.
\newblock A characterization of the {W}ishart distribution.
\newblock \emph{Ann. Math. Statist.}, 33:\penalty0 1272--1280, 1962.

\bibitem[Olkin and Rubin(1964)]{OlRu1964}
I.~Olkin and H.~Rubin.
\newblock Multivariate beta distributions and independence properties of the
  {W}ishart distribution.
\newblock \emph{Ann. Math. Statist}, 35:\penalty0 261--269, 1964.

\bibitem[Weso{\l}owski(2007)]{Wes2007L}
J.~Weso{\l}owski.
\newblock Multiplicative {C}auchy functional equation and the equation of
  ratios on the {L}orentz cone.
\newblock \emph{Studia Math.}, 179\penalty0 (3):\penalty0 263--275, 2007.

\end{thebibliography}

\end{document}